\documentclass[12pt]{amsart}

\usepackage{fullpage}

\usepackage{amsmath,amssymb,amsthm}
\usepackage{hyperref}

\newtheorem{thm}{Theorem}[section]
\newtheorem{cor}[thm]{Corollary}
\newtheorem{lem}[thm]{Lemma}
\newtheorem{rem}[thm]{Remark}

\numberwithin{equation}{section}

\title{Geometry-Driven Conditioning of Multivariate Vandermonde Matrices in High-Degree Regimes}

\author{Omer Friedland}
\address{Institut de Math\'ematiques de Jussieu, Sorbonne Universit\'e, 4 place Jussieu, 75005 Paris}
\email{omer.friedland@imj-prg.fr}

\author{Yosef Yomdin}
\address{Department of Mathematics, The Weizmann Institute of Science, Rehovot 76100, Israel.}
\email{yosef.yomdin@weizmann.ac.il}

\begin{document}

\begin{abstract}
We study multivariate monomial Vandermonde matrices $V_N(Z)$ with
arbitrary distinct nodes $Z=\{z_1,\dots,z_s\}\subset B_2^n$ in the
high-degree regime $N\ge s-1$. Introducing a projection-based geometric
statistic -- the \emph{max--min projection separation} $\rho(Z,j)$ and
its minimum $\kappa(Z)=\min_j\rho(Z,j)$ -- we construct Lagrange
polynomials $Q_j\in\mathcal P_N^n$ with explicit coefficient bounds
$$
\|Q_j\|_\infty \lesssim s\Bigl(\frac{4n}{\rho(Z,j)}\Bigr)^{s-1}.
$$
These polynomials yield quantitative distance-to-span estimates for the
rows of $V_N(Z)$ and, as consequences,
$$
\sigma_{\min}(V_N(Z)) \gtrsim 
\frac{\kappa(Z)^{s-1}}{(4n)^{s-1} s\sqrt{s \nu(n,N)}},
\quad
\nu(n,N)={N+n\choose N},
$$
and an explicit right inverse $V_N(Z)^+$ with operator-norm control
$$
\|V_N(Z)^+\| \lesssim 
s^{3/2}\sqrt{\nu(n,N)}\Bigl(\frac{4n}{\kappa(Z)}\Bigr)^{s-1}.
$$
Our estimates are dimension-explicit and expressed directly in terms of
the local geometry parameter $\kappa(Z)$; they apply to \emph{every}
distinct node set $Z\subset B_2^n$ without any \emph{a priori} separation
assumptions. In particular, $V_N(Z)$ has full row rank whenever
$N\ge s-1$. The results complement the Fourier-type theory
(on the complex unit circle/torus), where lower bounds for
$\sigma_{\min}$ hinge on uniform separation or cluster structure; here
stability is quantified instead via high polynomial degree and the
projection geometry of $Z$.
\end{abstract}

\maketitle

\section{Introduction}

Multivariate polynomial interpolation is fundamental in approximation theory and
scientific computing. Given $Z=\{z_1,\dots,z_s\}\subset B_2^n\subset\mathbb{R}^n$ (for $s\ge 2$) and a degree bound $N$, the \emph{monomial} Vandermonde matrix
$V_N(Z)\in\mathbb{R}^{s\times \nu}$ with $\nu=\binom{N+n}{N}$ collects, in its
$i$-th row, the evaluations $ (z_i^\alpha )_{|\alpha|\le N}$. In high
dimensions the basis size grows quickly with $N$ and geometric relations among
the nodes may cause near dependencies, so the stability of $V_N(Z)$ is delicate.
Much of the literature on Fourier/torus matrices ensures stability via global
\emph{separation} or controlled cluster models; here we aim for \emph{separation‑free}
guarantees in the monomial basis.

\subsubsection*{Standing assumptions and notation.}

We use the graded monomial basis $\{z^\alpha:|\alpha|\le N\}$, multi‑indices
$\alpha=(\alpha_1,\dots,\alpha_n)\in\mathbb{Z}_{\ge0}^n$ with
$|\alpha|=\sum_{\ell=1}^n\alpha_\ell$ and $z^\alpha=\prod_{\ell=1}^n z_\ell^{\alpha_\ell}$,
and write $P(z)=\sum_{|\alpha|\le N}a_\alpha z^\alpha$ with the coefficient norm
$\|P\|_\infty=\max_{|\alpha|\le N}|a_\alpha|$. The space $\mathcal{P}_N^n$ has
dimension $\nu(n,N)=\binom{N+n}{N}$. The evaluation map $z\mapsto V_N(z)\in\mathbb{R}^{\nu}$
is the degree‑$N$ Veronese embedding; we use this only to fix the basis and dimension.
We normalize $Z\subset B_2^n$; a global rescaling of the ambient space merely rescales
columns of $V_N(Z)$ and only changes constants in our bounds.

\subsubsection*{Classical square case vs. high‑degree regime.}

Since $\nu(n,N)\sim N^n$, unique interpolation is the generic situation when
$s\approx \nu(n,N)$; ill‑posedness in the square case $s=\nu(n,N)$ occurs precisely
when $Z$ lies in the zero set of a nonzero polynomial of degree $\le N$
(equivalently, the rows of $V_N(Z)$ are linearly dependent). In this paper we
instead study the \emph{high‑degree} regime $N\ge s-1$, where $\nu(n,N)\gg s$ and
interpolation is non‑unique. Our question is: can one obtain explicit,
dimension‑dependent stability bounds for $V_N(Z)$ that hold for \emph{all} distinct
node sets, without any global separation assumptions?

\subsubsection*{Local projection geometry.}

For $j\in\{1,\dots,s\}$ define the \emph{max–min projection separation}
$$
\rho(Z,j)=\max_{\|v\|_2=1}\min_{i\ne j} |\langle v,z_j-z_i\rangle |,\quad
\kappa(Z)=\min_{1\le j\le s}\rho(Z,j) > 0.
$$
This statistic controls both the coefficients of suitable Lagrange polynomials and the
geometric separation of the rows of $V_N(Z)$ from the spans of the others.

\begin{rem}[Positivity of the projection separation]\label{rem:kappa-positive}
In the definition
$$
\rho(Z,j)
=\max_{\|v\|_2=1}\,\min_{i\ne j}|\langle v,z_j-z_i\rangle|,
\quad
\kappa(Z)=\min_{1\le j\le s}\rho(Z,j),
$$
we implicitly use that $\rho(Z,j)>0$ for each $j$, and hence $\kappa(Z)>0$, as soon
as the nodes $z_1,\dots,z_s$ are distinct. This is an immediate consequence of a
simple geometric observation.

Fix $j$ and set $u_i=z_j-z_i\in\mathbb{R}^n$ for $i\ne j$. Since the nodes are distinct,
each $u_i\neq 0$. Consider the continuous function
$$
f:S^{n-1}\to[0,\infty),\quad
f(v):=\min_{i\ne j}|\langle v,u_i\rangle|.
$$
Here $S^{n-1}=\{v\in\mathbb{R}^n:\|v\|_2=1\}$ is the unit sphere. For each $i\ne j$ the
zero set
$$
H_i:=\{v\in S^{n-1}:\langle v,u_i\rangle=0\}
$$
is a proper ``great subsphere'' of codimension~$1$, and hence has empty interior in
$S^{n-1}$. Consequently, the finite union $\bigcup_{i\ne j}H_i$ also has empty interior
and cannot equal the whole of $S^{n-1}$. Thus there exists $v\in S^{n-1}$ with
$\langle v,u_i\rangle\neq 0$ for all $i\ne j$, so
$$
f(v)=\min_{i\ne j}|\langle v,u_i\rangle|>0.
$$
Since $f$ is continuous and $S^{n-1}$ is compact, $f$ attains its maximum, and we
obtain
$$
\rho(Z,j)=\max_{\|v\|_2=1}f(v)\ge f(v)>0.
$$
Taking the minimum over $j$ shows that $\kappa(Z)=\min_j\rho(Z,j)>0$ whenever the
points of $Z$ are pairwise distinct. In particular, all expressions involving
$\kappa(Z)$ in the sequel are well-defined.
\end{rem}

\subsubsection*{Our contributions (informal).}

In the regime $N\ge s-1$ and for \emph{arbitrary} distinct $Z\subset B_2^n$, we prove:
\begin{itemize}
\item[(1)] \textbf{Coefficient‑controlled Lagrange polynomials.}
For each $j$ there exists $Q_j\in\mathcal{P}_N^n$ with $Q_j(z_i)=\delta_{ij}$ and
$$
\|Q_j\|_\infty \le s (\frac{4n}{\rho(Z,j)} )^{s-1}.
$$
\item[(2)] \textbf{Row–to–span separation and angles.}
If $W_j=\mathrm{span}\{V_N(z_i):i\ne j\}$, then
$$
\mathrm{dist} (V_N(z_j),W_j ) \ge\
\frac{\rho(Z,j)^{ s-1}}{(4n)^{ s-1} s \sqrt{\nu}},\quad
\sin\theta_j \ge \frac{\rho(Z,j)^{ s-1}}{(4n)^{ s-1} s \nu}.
$$
\item[(3)] \textbf{Spectral stability and an explicit right inverse.}
$$
\sigma_{\min} (V_N(Z) ) \ge\
\frac{\kappa(Z)^{ s-1}}{(4n)^{ s-1} s \sqrt{s \nu}},\quad
\exists V_N(Z)^+: \|V_N(Z)^+\| \le\
s^{3/2}\sqrt{\nu} (\frac{4n}{\kappa(Z)} )^{s-1}.
$$
In particular, $V_N(Z)$ has full row rank whenever $N\ge s-1$.
\end{itemize}
We make no \emph{a priori} separation assumptions; instead, our bounds
hold for every distinct node set, with constants expressed explicitly in
terms of the geometry parameter $\kappa(Z)$.

\subsubsection*{Relation to prior work.}

On the complex unit circle/torus, sharp lower bounds for the smallest singular value of
(partial) Fourier/Vandermonde matrices typically require \emph{separation} or a controlled
\emph{cluster} model; see, e.g.,
\cite{Moitra2015,KunisNagel2020LAA,KunisNagel2021NA,KunisNagelStrotmann2021,
BatenkovDemanetGoldmanYomdin2020,BatenkovDiederichsGoldmanYomdin2021,LiLiao2021}.
Classical conditioning results for monomial Vandermonde matrices (univariate / nodes in the
unit disk) include \cite{Gautschi1975,Pan2016,Bazan1999,AubelBoelcskei2019}.
Our bounds address a different regime: multivariate, separation-free, and
high-degree, with coefficient-controlled Lagrange polynomials and a
concrete right inverse in the monomial basis. In particular, we do not
assume any uniform lower bound on separation; all dependence on the
geometry of $Z$ enters explicitly through $\kappa(Z)$.

\subsubsection*{Organization.}

Section~\ref{sec:prelim} establishes the main technical tools: coefficient bounds
for univariate and multivariate Lagrange polynomials and the resulting
row-to-span and rank properties of $V_N(Z)$. 
Section~\ref{sec:spectral} applies these geometric estimates to the spectral
analysis, providing bounds on the smallest and largest singular values and hence
on the condition number of $V_N(Z)$.

\section{Preliminaries}\label{sec:prelim}

In this section, we establish foundational results that underpin our analysis of high-dimensional polynomial interpolation. We begin with coefficient bounds for univariate polynomials, which serve as a cornerstone for extending our findings to the multivariate setting. We then explore multivariate interpolating polynomials, leading to key insights into the geometric properties of the Vandermonde matrix and the interpolation problem in higher dimensions.

\subsection{Coefficient Bounds for Univariate Polynomials}

Let $\{t_1, t_2, \dots, t_s\}\subset [-1, 1] \subset \mathbb{R} $ be a set of distinct points with minimal separation $\kappa $:
$$
|t_i - t_j| \geq \kappa > 0 \quad \text{for all } i \neq j.
$$

Consider the one-dimensional interpolating polynomial $ p_y(t) $ of degree $ s - 1 $ that interpolates the data points $\{(t_i, y_i)\} $, where $ y = (y_1, y_2, \dots, y_s) $ is given. The polynomial $ p_y(t) $ can be expressed in the monomial basis as:
$$
p_y(t) = \sum_{k=0}^{s-1} c_k t^k.
$$

We establish the following bounds on the coefficients $ c_k $:

\begin{lem}\label{lem:univariate}
Under the above conditions, the coefficients of $ p_y(t) $ satisfy:
$$
|c_k| \leq \| y \|_\infty \cdot (s - k) \left( \frac{4}{\kappa}\right)^{s-1}.
$$
Moreover, the sum of the absolute values of the coefficients is bounded by:
$$
\sum_{k=0}^{s-1} |c_k| \leq \frac{s(s+1)}{2}\| y \|_\infty \left( \frac{4}{\kappa}\right)^{s-1}.
$$
\end{lem}

These bounds reflect how the separation of interpolation points influences the stability and conditioning of the interpolation problem.

\begin{proof}
The polynomial $ p_y(t) $ can be represented in Newton's form:
$$
p_y(t) = \sum_{j=0}^{s-1}\Delta_j(t_1, t_2, \ldots, t_{j+1}) \prod_{i=1}^{j} (t - t_i),
$$
where $\Delta_j(t_1, t_2, \ldots, t_{j+1}) $ denotes the $ j $-th order divided difference for the data points $\{t_1, t_2, \ldots, t_{j+1}\} $. Specifically, $\Delta_0(t_i) = y_i $, and for $ j \geq 1 $,
$$
\Delta_j(t_1, t_2, \ldots, t_{j+1}) = \frac{\Delta_{j-1}(t_2, t_3, \ldots, t_{j+1}) - \Delta_{j-1}(t_1, t_2, \ldots, t_j)}{t_{j+1} - t_1}.
$$

To connect the Newton form to the monomial form, we expand each term:
$$
\prod_{i=1}^{j} (t - t_i) = t^j + b_{j, j-1} t^{j-1} + b_{j, j-2} t^{j-2} + \cdots + b_{j, 0},
$$
where $ b_{j, k} $ is the coefficient of $ t^k $ in this expansion. The coefficient $ c_k $ of $ t^k $ in the monomial form is obtained by summing up the contributions from all terms in the Newton expansion that include $ t^k $:
$$
c_k = \sum_{j=k}^{s-1}\Delta_j(t_1, t_2, \ldots, t_{j+1}) \cdot b_{j, k}.
$$

The coefficients $ b_{j, k} $ can be explicitly described as:
$$
b_{j, k} = (-1)^{j-k}\sum_{1 \leq i_1 < i_2 < \cdots < i_{j-k}\leq j} t_{i_1} t_{i_2}\cdots t_{i_{j-k}}.
$$

Since each $ t_i \in [-1, 1] $, we have $ |t_i| \leq 1 $ for all $ i $. Therefore,
$$
|b_{j, k}| \leq \binom{j}{k},
$$
as the magnitude of any product of $ j - k $ terms is at most 1, and there are $\binom{j}{k} $ such products.

We proceed to bound $ |\Delta_j(t_1, t_2, \dots, t_{j+1})| $. For $ j = 0 $, $ |\Delta_0(t_i)| = |y_i| \leq \| y \|_\infty $. For $ j = 1 $,
$$
|\Delta_1(t_i, t_j)| = \left| \frac{y_j - y_i}{t_j - t_i}\right| \leq \frac{|y_j| + |y_i|}{|t_j - t_i|}\leq \frac{2 \| y \|_\infty}{\kappa},
$$
since $ |t_j - t_i| \geq \kappa $ and $ |y_i|, |y_j| \leq \| y \|_\infty $. By induction, for $ j = 0, 1, \dots, s-1 $, we have:
$$
|\Delta_j(t_1, t_2, \dots, t_{j+1})| \leq \frac{2^j \| y \|_\infty}{\kappa^j}.
$$

Combining the bounds on $ |\Delta_j| $ and $ |b_{j, k}| $, we have:
$$
|c_k| \leq \sum_{j=k}^{s-1}\frac{2^j \| y \|_\infty}{\kappa^j}\cdot \binom{j}{k}.
$$

Noting that $\binom{j}{k}\leq 2^j $ for $ 0 \leq k \leq j $, we obtain:
$$
|c_k| \leq \| y \|_\infty \sum_{j=k}^{s-1}\left( \frac{4}{\kappa}\right)^j.
$$

Since $\frac{4}{\kappa} > 1 $, the series grows geometrically, and we can bound it by:
$$
|c_k| \leq \| y \|_\infty \cdot (s - k) \left( \frac{4}{\kappa}\right)^{s-1}.
$$

Therefore,
$$
\sum_{k=0}^{s-1} |c_k| \leq \| y \|_\infty \left( \frac{4}{\kappa}\right)^{s-1}\sum_{k=0}^{s-1} (s - k)
= \frac{s(s+1)}{2}\| y \|_\infty \left( \frac{4}{\kappa}\right)^{s-1}.
$$
\end{proof}

\subsection{Extension to Multivariate Interpolating Polynomials}

Building on the univariate case, we consider multivariate interpolation. Let $ Z = \{ z_1, z_2, \dots, z_s \}\subset B_2^n \subset \mathbb{R}^n $ be a set of $ s $ distinct points within the unit ball, and let $ N \geq s - 1 $. Our goal is to construct interpolating polynomials $ Q_j \in \mathcal{P}_N^n $ satisfying:
$$
Q_j(z_i) = \delta_{ij}\quad \text{for } i = 1, \dots, s,
$$
where $\delta_{ij} $ is the Kronecker delta function.

We establish the following bound on the coefficients of $ Q_j $:

\begin{lem}\label{lem:Qj}
For each $j\in\{1,\dots,s\}$ there exists $Q_j\in\mathcal{P}_N^n$ with
$Q_j(z_i)=\delta_{ij}$ and whose coefficient vector in the monomial basis
satisfies
$$
\| Q_j \|_\infty \leq s \left( \frac{4 n}{\rho(Z, j)}\right)^{s - 1}.
$$
\end{lem}

\begin{proof}
Fix $j$ and choose a unit vector $v\in\mathbb{R}^n$ such that
$|\langle v,z_j-z_i\rangle|\ge \rho(Z,j)$ for all $i\ne j$.
Let $t_i=\langle v,z_i\rangle\in[-1,1]$ and define the univariate Lagrange
basis polynomial
$$
p_j(t)=\prod_{i\ne j}\frac{t-t_i}{t_j-t_i}=\sum_{k=0}^{s-1}a_{j,k} t^k. 
$$
Because $|t_j-t_i|\ge\rho(Z,j)$ we have
$\prod_{i\ne j}|t_j-t_i|\ge \rho(Z,j)^{s-1}$.
Expanding the numerator $\prod_{i\ne j}(t-t_i)=\sum_{k=0}^{s-1}A_k t^k$,
the coefficients are $A_k=(-1)^{s-1-k}e_{s-1-k}(t_{-j})$, where $e_m$
is the degree-$m$ elementary symmetric sum of $\{t_i\}_{i\ne j}$.
Since $|t_i|\le 1$, $|e_m|\le \binom{s-1}{m}$, hence
$$
|a_{j,k}|=\frac{|A_k|}{\prod_{i\ne j}|t_j-t_i|}
\le \frac{\binom{s-1}{k}}{\rho(Z,j)^{s-1}}
\le \frac{2^{ s-1}}{\rho(Z,j)^{s-1}}\quad(0\le k\le s-1).
$$

Now set $Q_j(z):=p_j(\langle v,z\rangle)$, so $Q_j(z_i)=\delta_{ij}$ and
$\deg Q_j\le s-1\le N$.
By the multinomial theorem,
$$
Q_j(z)=\sum_{k=0}^{s-1} a_{j,k} (\sum_{\ell=1}^n v_\ell z_\ell )^k
 =\sum_{|\alpha|\le s-1} c_{j,\alpha} z^\alpha,\quad
c_{j,\alpha}=a_{j,k}\binom{k}{\alpha}v^\alpha, k=|\alpha|.
$$
Because $|v_\ell|\le 1$ we have $|v^\alpha|\le 1$, and for any fixed $k$,
$$
\binom{k}{\alpha}\le \sum_{|\beta|=k}\binom{k}{\beta}=n^k. 
$$
Therefore, for $k=|\alpha|\le s-1$,
$$
|c_{j,\alpha}|\le |a_{j,k}| \binom{k}{\alpha}
\le \frac{\binom{s-1}{k} n^k}{\rho(Z,j)^{s-1}}
\le \frac{2^{ s-1} n^{ s-1}}{\rho(Z,j)^{s-1}}
= (\frac{2n}{\rho(Z,j)} )^{s-1}.
$$
Taking the maximum over $|\alpha|\le N$ yields
$$
\|Q_j\|_\infty \le (\frac{2n}{\rho(Z,j)} )^{s-1}
\le s (\frac{4n}{\rho(Z,j)} )^{s-1},
$$
since $1\le s 2^{ s-1}$ for all $s\ge 1$.
\end{proof}

\subsection{Geometric Properties of the Vandermonde Matrix}

The Vandermonde matrix $ V_N(Z) $ plays a central role in analyzing the interpolation problem. Understanding the geometric relationships between its rows and the subspaces they span is crucial for assessing the conditioning and stability of the interpolation. The following result indicates that the rows of the Vandermonde matrix are well-separated from the subspaces spanned by the other rows, which is essential for the invertibility and conditioning of $ V_N(Z) $.

\begin{thm}\label{thm:distance}
Let $ N \geq s - 1 $. The distance between the row $ V_N(z_j) $ and the subspace
$ W_j $ spanned by $\{ V_N(z_i) : i \neq j \} $ satisfies
$$
\operatorname{dist}(V_N(z_j), W_j) \geq
\frac{\rho(Z, j)^{s - 1}}{(4 n)^{s - 1} s \sqrt{\nu}}.
$$
\end{thm}

\begin{proof}
Take the interpolating polynomial $Q_j$ of Lemma \ref{lem:Qj}. It satisfies 
$$
Q_j(z_i) = \delta_{ij}\quad \text{for } i = 1, \dots, s,
$$
where $\delta_{ij} $ is the Kronecker delta. Let $ c_j \in \mathbb{R}^\nu $ be the
coefficient vector of $ Q_j $ in the monomial basis:
$$
Q_j(z) = \sum_{|\alpha| \leq N} c_{j,\alpha} z^\alpha.
$$

The interpolation conditions can be written in matrix form as
$$
V_N(Z) c_j = e_j,
$$
where $ e_j $ is the $ j $-th standard basis vector in $\mathbb{R}^s $. Thus,
for every $i\neq j$ we have $\langle c_j, V_N(z_i)\rangle = 0$, so $ c_j $ is
orthogonal to $ W_j := \mathrm{span}\{V_N(z_i) : i\neq j\}$.

Since $c_j$ is orthogonal to $W_j$, for any vector $x\in\mathbb{R}^\nu$ we obtain
$$
\operatorname{dist}(x,W_j) \ge \frac{|\langle c_j,x\rangle|}{\|c_j\|_2}.
$$
Applying this to $x=V_N(z_j)$ and using $\langle c_j,V_N(z_j)\rangle = 1$ gives
$$
\operatorname{dist}(V_N(z_j), W_j) \ge \frac{1}{\| c_j \|_2}.
$$

From Lemma \ref{lem:Qj}, we also have the coefficient bound
$$
\| Q_j \|_\infty = \max_{|\alpha| \leq N} | c_{j,\alpha} |
\leq s \left( \frac{4 n}{\rho(Z, j)}\right)^{s - 1}.
$$
Since $ c_j \in \mathbb{R}^\nu $, its Euclidean norm can be bounded by
$$
\| c_j \|_2 \leq \sqrt{\nu} \| c_j \|_\infty,
$$
where $\nu = \nu(n, N) = \binom{N + n}{N} $ is the number of monomials of
degree at most $ N $. Thus,
$$
\| c_j \|_2 \leq \sqrt{\nu}\cdot s \left( \frac{4 n}{\rho(Z, j)}\right)^{s - 1}.
$$

Combining this with the previous inequality, we obtain
$$
\operatorname{dist}(V_N(z_j), W_j)
 \geq \frac{1}{\|c_j\|_2}
 \geq \frac{\rho(Z, j)^{s - 1}}{(4 n)^{s - 1} s \sqrt{\nu}},
$$
which is the desired estimate.
\end{proof}

Let $ U $ be a subspace, and let $\mathbf{x} $ be a non-zero vector. The angle $\theta $ between the vector $\mathbf{x} $ and the subspace $ U $ is defined as the smallest angle between $\mathbf{x} $ and any non-zero vector in $ U $. Leveraging the concept of orthogonal projection, the angle can be expressed in terms of the distance from $\mathbf{x} $ to $ U $:
$$
\theta = \arcsin\left( \frac{ \operatorname{dist}(\mathbf{x}, U) }{ \| \mathbf{x}\|_2 }\right),
$$
where $\operatorname{dist}(\mathbf{x}, U) $ is the Euclidean distance from $\mathbf{x} $ to the subspace $ U $.

\begin{cor}\label{cor:angle}
The angle $\theta_j $ between the row $ V_N(z_j) $ and the subspace $ W_j $ satisfies:
$$
\sin \theta_j \geq \frac{\rho(Z, j)^{s - 1}}{(4 n)^{s - 1} s \nu}.
$$
\end{cor}

The lower bound on $\sin \theta_j $ ensures that the rows are not nearly parallel to the subspaces spanned by the others, which would otherwise lead to ill-conditioning.

\begin{proof}
Note that we have: 
$$
\sin \theta_j = \frac{ \operatorname{dist}(V_N(z_j), W_j) }{ \| V_N(z_j) \|_2 }.
$$

From Theorem \ref{thm:distance}, we have the lower bound:
$$
\operatorname{dist}(V_N(z_j), W_j) \geq \frac{\rho(Z, j)^{s - 1}}{(4 n)^{s - 1} s \sqrt{\nu}}.
$$

Each entry of $ V_N(z_j) $ is $ (V_N(z_j))_i = z_j^{\alpha_i} $, where $\alpha_i$ corresponds to the monomial associated with $ V_N(z_j) $, and $ z_j \in B_2^n $ satisfies $\|z_j\|_2 \leq 1 $. Therefore, $ | (V_N(z_j))_i | \leq 1 $ for all $ i $. The Euclidean norm of $ V_N(z_j) $ is thus bounded by:
$$
\| V_N(z_j) \|_2 = \left( \sum_{i=1}^\nu | (V_N(z_j))_i |^2 \right)^{1/2}\leq \sqrt{\nu}.
$$

Substituting the bounds into the expression for $\sin \theta_j $, we obtain:
$$
\sin \theta_j \geq \frac{\rho(Z, j)^{s - 1}}{(4 n)^{s - 1} s \nu}.
$$
\end{proof}

\subsection{Existence and Bounds for Interpolating Polynomials}

We now establish the existence of interpolating polynomials in $\mathcal{P}_N^n $ with bounded coefficients, emphasizing the role of geometric separation in controlling the norms. The following result assures that interpolation is always possible under the given conditions, and the norms of the interpolating polynomials are controlled by the geometric properties of $ Z $.

\begin{thm}\label{thm:existence}
Let $ Z \subset B_2^n $ be a set of $ s $ distinct points with minimal projection
separation $\kappa(Z) > 0 $, and let $ N \geq s - 1 $. Then, for any data
$ y = (y_1, \dots, y_s) $, there exists a polynomial $ P \in \mathcal{P}_N^n $
satisfying $ P(z_i) = y_i $ for $ i = 1, \dots, s $, with
$$
\|P\|_\infty \leq s^2 \| y \|_\infty \left( \frac{4n}{\kappa(Z)}\right)^{s-1}.
$$
Moreover, all such polynomials form an affine subspace $ L(Z, y) $ of dimension at least
$$
\nu(n, N) - s.
$$
\end{thm}

\begin{proof}
Since $N \ge s-1$, Lemma~\ref{lem:Qj} provides, for each $j\in\{1,\dots,s\}$,
a polynomial $Q_j\in\mathcal{P}_N^n$ such that
$$
Q_j(z_i) = \delta_{ij}\quad\text{for } i=1,\dots,s,
$$
and whose coefficient vector in the monomial basis satisfies
$$
\|Q_j\|_\infty \le s\Bigl(\frac{4n}{\rho(Z,j)}\Bigr)^{s-1}.
$$
Since $\kappa(Z) = \min_{1\le j\le s} \rho(Z,j)$, we have
$$
\|Q_j\|_\infty \le s\Bigl(\frac{4n}{\kappa(Z)}\Bigr)^{s-1}
\quad\text{for all }j=1,\dots,s.
$$

Given data $y=(y_1,\dots,y_s)$, define
$$
P(z) := \sum_{j=1}^s y_j Q_j(z).
$$
Then for each $i=1,\dots,s$,
$$
P(z_i) = \sum_{j=1}^s y_j Q_j(z_i)
 = \sum_{j=1}^s y_j \delta_{ij}
 = y_i,
$$
so $P$ is an interpolating polynomial.

Write $Q_j(z) = \sum_{|\alpha|\le N} c_{j,\alpha} z^\alpha$ and
$P(z) = \sum_{|\alpha|\le N} a_\alpha z^\alpha$. Then
$$
a_\alpha = \sum_{j=1}^s y_j c_{j,\alpha},
$$
and hence
$$
|a_\alpha|
 \le \sum_{j=1}^s |y_j| |c_{j,\alpha}|
 \le \|y\|_\infty \sum_{j=1}^s |c_{j,\alpha}|
 \le \|y\|_\infty \sum_{j=1}^s \|Q_j\|_\infty.
$$
Using the uniform bound on $\|Q_j\|_\infty$ gives
$$
|a_\alpha|
 \le \|y\|_\infty \cdot s \cdot
 \max_{1\le j\le s}\|Q_j\|_\infty
 \le \|y\|_\infty \cdot s \cdot
 s\Bigl(\frac{4n}{\kappa(Z)}\Bigr)^{s-1}
 = s^2 \|y\|_\infty\Bigl(\frac{4n}{\kappa(Z)}\Bigr)^{s-1}.
$$
Taking the maximum over all $|\alpha|\le N$ yields
$$
\|P\|_\infty = \max_{|\alpha|\le N} |a_\alpha|
 \le s^2 \|y\|_\infty\Bigl(\frac{4n}{\kappa(Z)}\Bigr)^{s-1},
$$
which proves the stated coefficient bound.

For the dimension statement, let $V_N(Z)$ denote the Vandermonde matrix.
The map
$$
\mathcal{P}_N^n \longrightarrow \mathbb{R}^s,\quad
P \longmapsto (P(z_1),\dots,P(z_s))
$$
is linear, and in the monomial basis its matrix is $V_N(Z)$.
Thus the set of interpolants
$$
L(Z,y) := \{P\in\mathcal{P}_N^n : P(z_i)=y_i, i=1,\dots,s\}
$$
is either empty or an affine translate of $\ker V_N(Z)$.
We have already exhibited one interpolant $P$, so $L(Z,y)\neq\emptyset$
and in fact $L(Z,y)$ is an affine subspace with
$$
\dim L(Z,y) = \dim\ker V_N(Z)
 = \nu(n,N) - \operatorname{rank} V_N(Z).
$$
Since $V_N(Z)$ has $s$ rows, $\operatorname{rank} V_N(Z)\le s$, and hence
$$
\dim L(Z,y) \ge \nu(n,N) - s.
$$
This completes the proof.
\end{proof}

\subsection{Full Rank and Invertibility of the Vandermonde Matrix}

Understanding the rank and invertibility of $ V_N(Z) $ is vital for ensuring the solvability of the interpolation problem.

\begin{lem}[Full row rank]\label{lem:full-rank}
Let $Z=\{z_1,\dots,z_s\}\subset B_2^n$ be $s$ distinct points and let $N\ge s-1$.
Then the Vandermonde matrix $V_N(Z)$ has full row rank:
$$
\operatorname{rank}(V_N(Z))=s.
$$
In particular, the rows of $V_N(Z)$ are linearly independent.
\end{lem}

\begin{proof}
By Lemma~\ref{lem:Qj}, for each $j\in\{1,\dots,s\}$ there exists a polynomial
$Q_j\in\mathcal P_N^n$ with $Q_j(z_i)=\delta_{ij}$. Let $c_j\in\mathbb R^\nu$
denote the coefficient vector of $Q_j$ in the monomial basis, and set
$C=[ c_1 \cdots c_s ]\in\mathbb R^{\nu\times s}$. Then
$$
V_N(Z) C
= [ V_N(Z)c_1 \cdots V_N(Z)c_s ]
= [ e_1 \cdots e_s ]
=I_s.
$$
Hence $V_N(Z)$ admits a right inverse and therefore
$\operatorname{rank}(V_N(Z))=s$. This also implies that its $s$ rows are
linearly independent.
\end{proof}

Next, we study the right inverse of the vandermonde matrix. The next result confirms that the interpolation problem is well-posed in terms of solvability.

\begin{thm}\label{thm:rightinverse}
Let $Z=\{z_1,\dots,z_s\}\subset B_2^n$ consist of $s$ distinct points and let
$N\ge s-1$. Then there exists a right inverse $ V_N(Z)^+ \in \mathbb{R}^{\nu \times s} $ such that
$$
V_N(Z) \cdot V_N(Z)^+ = I_s,
$$
where $ I_s $ is the identity matrix of size $ s \times s $. Moreover, we have
$$
\| V_N(Z)^+ \| \leq s^{3/2}\cdot \sqrt{\nu}\left( \frac{4n}{\kappa(Z)}\right)^{s - 1}.
$$
\end{thm}

\begin{proof}
For each $ j = 1, 2, \dots, s $, let $ Q_j \in \mathcal{P}_N^n $ be the interpolating polynomial provided by Lemma~\ref{lem:Qj}, so that
$$
Q_j(z_i) = \delta_{ij}\quad \text{for all } i = 1, 2, \dots, s,
$$
where $\delta_{ij} $ is the Kronecker delta function, i.e. $\delta_{ij} = 1 $ if $ i = j $ and $\delta_{ij} = 0 $ otherwise.

Express each $ Q_j $ in the monomial basis:
$$
Q_j(z) = \sum_{|\alpha| \leq N} c_{j,\alpha} z^\alpha,
$$
where $ c_{j,\alpha} $ are the coefficients corresponding to the multi-index $\alpha $.

Let $ c_j \in \mathbb{R}^\nu $ be the coefficient vector of $ Q_j $, such that:
$$
c_j = \begin{bmatrix}
c_{j,\alpha_1} & c_{j,\alpha_2} & \cdots & c_{j,\alpha_\nu}
\end{bmatrix}^\top.
$$

Construct the matrix $ V_N(Z)^+ \in \mathbb{R}^{\nu \times s} $ by arranging the coefficient vectors $ c_j $ as its columns:
$$
V_N(Z)^+ = \begin{bmatrix}
c_1 & c_2 & \cdots & c_s
\end{bmatrix}.
$$

To verify that $ V_N(Z)^+ $ is indeed a right inverse of $ V_N(Z) $, compute the product $ V_N(Z) \cdot V_N(Z)^+ $:
$$
V_N(Z) \cdot V_N(Z)^+ = V_N(Z) \cdot \begin{bmatrix}
c_1 & c_2 & \cdots & c_s
\end{bmatrix} = \begin{bmatrix}
V_N(Z) c_1 & V_N(Z) c_2 & \cdots & V_N(Z) c_s
\end{bmatrix}.
$$

By the construction of the interpolating polynomials $ Q_j $, each $ c_j $ satisfies:
$$
V_N(Z) c_j = Q_j(Z) = \begin{bmatrix}
Q_j(z_1) \\
Q_j(z_2) \\
\vdots \\
Q_j(z_s)
\end{bmatrix} = \begin{bmatrix}
\delta_{j1}\\
\delta_{j2}\\
\vdots \\
\delta_{js}
\end{bmatrix} = e_j,
$$
where $ e_j $ is the $ j $-th standard basis vector in $\mathbb{R}^s $.

Substituting back, we obtain:
$$
V_N(Z) \cdot V_N(Z)^+ = \begin{bmatrix}
e_1 & e_2 & \cdots & e_s
\end{bmatrix} = I_s,
$$
where $ I_s $ is the $ s \times s $ identity matrix.

Thus, by constructing the right inverse $ V_N(Z)^+ $ using the interpolating polynomials $ Q_j $, we have established that:
$$
V_N(Z) \cdot V_N(Z)^+ = I_s.
$$

This confirms that $ V_N(Z)^+ $ is indeed a right inverse of the Vandermonde matrix $ V_N(Z) $, and consequently, $ V_N(Z) $ has full row rank.

The operator norm $\| V_N(Z)^+ \| $ is defined as
$$
\| V_N(Z)^+ \| = \sup_{\| y \|_2 = 1}\| V_N(Z)^+ y \|_2.
$$

Given that $ V_N(Z)^+ $ consists of the coefficient vectors $ c_j $, we can bound the norm $\| V_N(Z)^+ y \|_2 $ using the Cauchy-Schwarz inequality:
$$
\| V_N(Z)^+ y \|_2 = \left\| \sum_{j=1}^s y_j c_j \right\|_2 \leq \sum_{j=1}^s |y_j| \cdot \| c_j \|_2.
$$

Since $\| y \|_2 = 1 $, it follows that
$$
\sum_{j=1}^s |y_j| \leq \sqrt{s}\cdot \| y \|_2 = \sqrt{s}.
$$

Thus,
$$
\| V_N(Z)^+ \| \leq \sqrt{s}\cdot \max_{1 \leq j \leq s}\| c_j \|_2.
$$

By Lemma \ref{lem:Qj} and the definition of $\kappa(Z)$, we have
$$
\| Q_j \|_\infty \leq s \left( \frac{4n}{\rho(Z,j)}\right)^{s - 1}
\leq s \left( \frac{4n}{\kappa(Z)}\right)^{s - 1}.
$$

This bound on the $ L^\infty $ norm implies that each coefficient vector $ c_j $ satisfies
$$
\| c_j \|_\infty \leq s \left( \frac{4n}{\kappa(Z)}\right)^{s - 1}.
$$

Consequently, the $ L^2 $ norm of each $ c_j $ can be bounded by
$$
\| c_j \|_2 \leq \sqrt{\nu}\cdot \| c_j \|_\infty \leq s \cdot \sqrt{\nu}\left( \frac{4n}{\kappa(Z)}\right)^{s - 1}.
$$

Substituting this into the earlier inequality for $\| V_N(Z)^+ \| $, we obtain
$$
\| V_N(Z)^+ \| \leq \sqrt{s}\cdot s \cdot \sqrt{\nu}\left( \frac{4n}{\kappa(Z)}\right)^{s - 1} = s^{3/2}\cdot \sqrt{\nu}\left( \frac{4n}{\kappa(Z)}\right)^{s - 1}.
$$
\end{proof}

\section{Singular Values Analysis}\label{sec:spectral}

In this section, we analyze the singular values of the Vandermonde matrix $V_N(Z)$, which is crucial for understanding the conditioning of the interpolation problem. We establish both lower and upper bounds on the singular values and conclude with a bound on the condition number of $V_N(Z)$.

\subsection{Lower Bound on the Smallest Singular Value}

We begin with a general lemma that provides a lower bound on the smallest singular value of any matrix in terms of the distances of its columns to the subspaces spanned by the other columns.

\begin{lem}\label{lem:sigma_n}
Let $ M $ be an $ n \times m $ matrix with columns $ C_1, \dots, C_m $. For each $ j = 1, \dots, m $, let $ W_j $ be the subspace spanned by all columns except $ C_j $. Then,
$$
\sigma_{\min}(M) \geq \frac{1}{\sqrt{m}}\min_{1 \leq j \leq m}\operatorname{dist}(C_j, W_j),
$$
where $\sigma_{\min}(M) $ denotes the smallest singular value of $ M $, and $\operatorname{dist}(C_j, W_j) $ is the Euclidean distance from $ C_j $ to the subspace $ W_j $.
\end{lem}

\begin{proof}
Consider any vector $ x \in \mathbb{R}^m $ with $\| x \|_2 = 1 $. By the Cauchy-Schwarz inequality, there exists at least one index $ j \in \{1, \dots, m\} $ such that $ |x_j| \geq \frac{1}{\sqrt{m}} $.

Now, examine the product $ Mx $:
$$
Mx = \sum_{i=1}^m x_i C_i.
$$
Decompose this sum by isolating the $ j $-th term:
$$
Mx = x_j C_j + \sum_{\substack{i=1\ i \neq j}}^m x_i C_i.
$$
Let $ w_j = \sum_{\substack{i=1\ i \neq j}}^m x_i C_i $, which resides in the subspace $ W_j $. Therefore,
$$
Mx = x_j C_j + w_j.
$$
The Euclidean norm of $ Mx $ satisfies
$$
\| Mx \|_2 = \| x_j C_j + w_j \|_2 \geq |x_j| \cdot \operatorname{dist}(C_j, W_j).
$$
This inequality holds because $\operatorname{dist}(C_j, W_j) $ represents the minimal distance between $ C_j $ and any vector in $ W_j $, ensuring that the contribution from $ C_j $ cannot be entirely "absorbed" by $ w_j $.

Given that $ |x_j| \geq \frac{1}{\sqrt{m}} $, it follows that
$$
\| Mx \|_2 \geq \frac{1}{\sqrt{m}}\cdot \operatorname{dist}(C_j, W_j).
$$
Taking the infimum over all unit vectors $ x \in \mathbb{R}^m $, we obtain
$$
\sigma_{\min}(M) = \inf_{\| x \|_2 = 1}\| Mx \|_2 \geq \frac{1}{\sqrt{m}}\cdot \min_{1 \leq j \leq m}\operatorname{dist}(C_j, W_j).
$$

Thus, the smallest singular value of $ M $ satisfies the stated inequality.
\end{proof}

Applying the general lemma to the transposed Vandermonde matrix $V_N(Z)$, we obtain a lower bound on its smallest singular value.

\begin{cor}\label{cor:lower_sigma_s}
The smallest singular value $\sigma_s(V_N(Z)) $ satisfies
$$
\sigma_s(V_N(Z)) \geq \frac{\kappa(Z)^{s-1}}{(4n)^{s-1} s \sqrt{s \cdot \nu}}.
$$
\end{cor}

\begin{proof}
Let $ M = V_N(Z)^\top $, an $\nu \times s $ matrix (with columns which are the $C_i = V_N(z_i)^\top$). From Lemma \ref{lem:sigma_n}, we have
$$
\sigma_s(M) \geq \frac{1}{\sqrt{s}}\cdot \min_{1 \leq j \leq s}\operatorname{dist}(C_j, W_j).
$$

Next, applying Theorem \ref{thm:distance}, we established that for each $ j $
$$
\operatorname{dist}(C_j, W_j) \geq \frac{\kappa(Z)^{s-1}}{(4n)^{s-1} s \sqrt{\nu}},
$$
 where we used $\kappa(Z) \leq \rho(Z, j) $ for all $ j $.

Substituting this into the inequality from Lemma \ref{lem:sigma_n}, we obtain
$$
\sigma_s(M) \geq \frac{1}{\sqrt{s}}\cdot \frac{\kappa(Z)^{s-1}}{(4n)^{s-1} s \sqrt{\nu}} = \frac{\kappa(Z)^{s-1}}{(4n)^{s-1} s \sqrt{s \cdot \nu}}.
$$
\end{proof}

\begin{rem}
More generally, if $V$ has full row rank and $B$ is any right inverse of $V$
(i.e., $VB = I$), then
$$
 \sigma_{\min}(V) \ge \frac{1}{\|B\|}.
$$
One way to see this is to let $V^+$ denote the Moore--Penrose pseudoinverse
of $V$. Then $V^+V$ is the orthogonal projector onto the row space of $V$,
so in particular $\|V^+V\|=1$ and, for any right inverse $B$,
$$
 V^+ = (V^+V)B.
$$
Hence $\|V^+\| \le \|B\|$. Since $\|V^+\| = 1/\sigma_{\min}(V)$, this implies
$1/\|B\| \le \sigma_{\min}(V)$. Applied with $B = V_N(Z)^+$ from
Theorem~\ref{thm:rightinverse}, this reproduces the lower bound of
Corollary~\ref{cor:lower_sigma_s}.
\end{rem}

\subsection{Upper Bound on the Largest Singular Value}

We establish an upper bound on the largest singular value of $ V_N(Z) $, which, together with the lower bound, allows us to assess the conditioning of the matrix.

\begin{lem}\label{lem:upper_sigma_1}
Let $ Z = \{z_1, \dots, z_s\}\subset B_2^n \subset \mathbb{R}^n $ be a set of $ s $ distinct points such that $\kappa(Z) > 0 $, and let $ N \geq s - 1 $. Then, the largest singular value $\sigma_1(V_N(Z)) $ of the Vandermonde matrix $ V_N(Z) $ satisfies
$$
\sigma_1(V_N(Z)) \leq \sqrt{ s \cdot \nu }.
$$
\end{lem}

\begin{proof}
Recall that the operator norm $\| M \|_2 $ of a matrix $ M $ is equal to its largest singular value $\sigma_1(M) $. For the Vandermonde matrix $ V_N(Z) $, we can bound $\| V_N(Z) \|_2 $ using the Frobenius norm $\| V_N(Z) \|_F $, which satisfies
$$
\| V_N(Z) \|_2 \leq \| V_N(Z) \|_F.
$$
The Frobenius norm is given by
$$
\| V_N(Z) \|_F = \sqrt{ \sum_{j=1}^{\nu}\| C_j \|_2^2 },
$$
where $ C_j $ denotes the $ j $-th column of $ V_N(Z) $.

By the assumption $ Z \subset B_2^n $, each interpolation point $ z_i \in Z $ satisfies $\| z_i \|_2 \leq 1 $. Consequently, for any multi-index $\alpha $ with $ |\alpha| \leq N $, the monomial $ z_i^\alpha $ satisfies
$$
|z_i^\alpha| \leq \| z_i \|_\infty^{|\alpha|}\leq \| z_i \|_2^{|\alpha|}\leq 1.
$$
Thus, each entry of the column $ C_j $ satisfies
$$
\| C_j \|_\infty = \max_{1 \leq i \leq s} |(C_j)_i| \leq 1.
$$
Using the relationship between the $\ell_\infty $-norm and the $\ell_2 $-norm in $\mathbb{R}^s $,
$$
\| C_j \|_2 \leq \sqrt{s}\cdot \| C_j \|_\infty,
$$
we obtain
$$
\| C_j \|_2 \leq \sqrt{s}.
$$
Substituting this into the Frobenius norm expression, we have
$$
\| V_N(Z) \|_F = \sqrt{ \sum_{j=1}^{\nu}\| C_j \|_2^2 }\leq \sqrt{ \sum_{j=1}^{\nu} (\sqrt{s})^2 } = \sqrt{ \sum_{j=1}^{\nu} s } = \sqrt{ s \cdot \nu }.
$$
Therefore,
$$
\sigma_1(V_N(Z)) = \| V_N(Z) \|_2 \leq \| V_N(Z) \|_F \leq \sqrt{ s \cdot \nu }. 
$$
\end{proof}

\subsection{Condition Number of the Vandermonde Matrix}

The condition number $\kappa(V_N(Z)) $ of the Vandermonde matrix $ V_N(Z) $ is defined as the ratio of its largest singular value to its smallest singular value:
$$
\kappa(V_N(Z)) = \frac{\sigma_1(V_N(Z))}{\sigma_s(V_N(Z))}.
$$

The condition number reflects the sensitivity of the interpolation problem to perturbations in the data. This upper bound provides insight into how the geometry of $ Z $ and the degree $ N $ influence the conditioning of the Vandermonde matrix.

Combining the bounds on the singular values, we derive a bound on the condition number of $ V_N(Z) $.

\begin{thm}\label{thm:conditionnumber}
The condition number $\kappa(V_N(Z)) $ of the Vandermonde matrix satisfies:
$$
\kappa(V_N(Z)) \leq s^2 \cdot \nu \cdot \left( \frac{4n}{\kappa(Z)}\right)^{s-1}.
$$
\end{thm}

\begin{proof}
This is a direct consequence of Lemma \ref{lem:upper_sigma_1} and Corollary \ref{cor:lower_sigma_s}.
\end{proof}


\end{document}